\documentclass[11pt, reqno]{amsart}
\usepackage{bbm}
\usepackage{amsfonts}
\usepackage{amssymb}
\usepackage{amsfonts,mathrsfs,amsmath,amssymb}

\theoremstyle{plain}

\numberwithin{equation}{section}

\newtheorem{theo}{Theorem}[section]
\newtheorem{pro}[theo]{Proposition}
\newtheorem{lem}[theo]{Lemma}
\newtheorem{cor}[theo]{Corollary}
\newtheorem{rem}[theo]{Remark}
\theoremstyle{definition}
\newtheorem{defi}[theo]{Definition}

\def\ot{\otimes}
\def\om{\omega}

\def\al{\alpha}

\def\be{\beta}
\def \bee{\begin{equation}\label}

\def\ga{\gamma}

\def\Om{\Omega}

\def\de{\delta}
\def\pa{\partial}

\def\la{\lambda}

\def\lra{\longrightarrow}

\def\k{\kappa}

\def\d{\mathbbm{d}}
\def\cc{\mathbbm{c}}

\newcommand{\N}{{\mathcal N}}
\newcommand{\C}{{\mathcal C}}

\newcommand{\B}{{\mathcal B}}

\newcommand{\Z}{{\mathcal Z}}

\newcommand{\A}{{\mathcal A}}

\def\c{{\bf c }}

\def\qed{\hfill\mbox{$\Box$}}

\def\bN{{\mathbb N}}
\def\bZ{{\mathbb Z}}

\def\bC{{\mathbb C}}

\def\P{{\mathcal P}}

\def\C{{\mathscr C}}

\def\mh{\mathfrak{h}}
\def\b{\mathfrak{b}}

\def\Ht{\mbox{\rm ht}}

\def\Ind{\mbox{\rm Ind}\,}

\def\End{\text{\rm End}}

\def\C{{\mathbb C}}

\def\Z{{\mathbb Z}}

\def\N{{\mathbb N}}
\def\1{{\bf 1}}
\def\b1{\mathbbold{1}}
\def\k{{\bf k}}
\def \End{{\rm End}}

\def\H{\widehat{H}_{4}}
\def\NO{\mbox{\,$\circ\atop\circ$}\,}
\begin{document}
\title[]
{Representations of affine Nappi-Witten algebras}

\author[Bao]{Yixin Bao}
\address{Department of Mathematics, Shanghai Jiaotong University,
Shanghai 200240, PR China}\email{wintervanilla84@gmail.com}

\author[Jiang]{Cuipo Jiang$^{*}$}
\thanks{$*$ Corresponding author  supported in part by the NSFC (10931006, 10871125), the Innovation
Program of Shanghai Municipal Education Commission (11ZZ18)}

\address{Department of Mathematics, Shanghai Jiaotong University, Shanghai
200240, China} \email{cpjiang@sjtu.edu.cn}

\author[Pei]{Yufeng Pei$^{**}$}
\address{Department of Mathematics, Shanghai Normal University,
Shanghai 200234, China}\email{peiyufeng@gmail.com}
\thanks{$**$ Supported in part by the NSFC (11026042,11071068), the Innovation
Program of Shanghai Municipal Education Commission (11YZ85), the
Academic Discipline Project of Shanghai Normal University
(DZL803),and ZJNSF(Y6100148)}

\date{}

\begin{abstract}
In this paper, we study the representation theory for the affine Lie
algebra $\H$ associated to the Nappi-Witten model $H_{4}$. We
classify all the irreducible highest weight modules of $\H$.
Furthermore, we give a necessary and sufficient condition for each
$\H$-(generalized) Verma module  to be irreducible.  For reducible
ones, we characterize all the linearly independent singular vectors.
Finally, we construct Wakimoto type modules for these Lie algebras
and interpret this construction in terms of vertex operator algebras
and their modules.

\end{abstract}


\maketitle

\section{\bf Introduction}
Two-dimensional conformal field theories (CFTs) has many
applications both in physics and mathematics. Vertex operator
algebras \cite{Bo, DL, FLM, LL, Xu}  provide a powerful algebraic
tool to study the general structure of conformal field theory as
well as various specific models \cite{D, FZ, W} etc. One of the
richest classes of CFTs consists of the Wess-Zumino-Novikov-Witten
(WZNW) models \cite{Witten}, which were studied originally within
the framework of semsimple(abelian) groups. It can be shown that
most properties of these models are substantially captured by the
properties of the corresponding vertex operator algebras \cite{FZ}.
For non-reductive groups,  few results on WZNW models are known.
However, there has been a great interest in WZNW models based on
nonabelian nonsemisimple Lie groups \cite{NW,Lian,KK} in the early
1990, partly because they allow the construction of exact string
backgrounds. Nappi and Witten showed in \cite{NW} that a WZNW
model(NW model) based on a central extension of the two-dimensional
Euclidean group describes the homogeneous four-dimensional
space-time corresponding to a gravitational plane wave. The
corresponding Lie algebra $H_4$ is called Nappi-Witten algebra. Just
as the non-twisted affine Kac-Moody Lie algebras given in
\cite{Kac}, the affine Nappi-Witten algebra $\H$ is defined to be
the central extension of the loop algebra of $H_4$. The study of the
representation theory of $\H$ was started in \cite{KK}. Further
studies on NW model were presented in \cite{CFS,DQ1,DQ2,DK}.

In the present paper,  we initiate a systematic study of
representations of the affine Nappi-Witten algebra $\H$. We discuss
the structure of all  the (generalized) Verma modules and their
irreducible quotients. We  give a necessary and sufficient condition
for each Verma module to be irreducible. Furthermore, for the
reducible ones , we obtain a complete description for the linearly
independent singular vectors and classify all the irreducible
highest weight modules. After that we construct a simple vertex
operator algebra associated to the NW model and classify all
irreducible representations. It is known that the Wakimoto modules
for affine Kac-Moody Lie algebras have important applications in
conformal field theory, representation theory, integrable systems,
and many other areas of mathematics and physics(see \cite{FF,Sz,Wa}
and references therein). In this paper, we also construct a class of
Wakimoto type modules (free field realizations) for the affine NW
algebra $\H$ and interpret this construction in terms of vertex
(operator) algebras and their modules.

This paper is organized as follows. In section 2, we recall some
basic results on the NW algebras and their affinizations. In section
3, we investigate the structure of the (generalized) Verma modules
over the affine NW algebras. In section 4, we discuss the vertex
operator algebra structures associated to the representations of
affine NW algebras. In section 5, by the vertex algebras and their
modules obtained in section 4, we construct a class of Wakimoto type
modules (free field realizations) over the affine NW algebra.

Throughout the paper, we  use $\bC, \bZ$, $\bZ_{+}$ and $\bN$ to
denote the sets of the complex numbers, the integers, the positive
integers and the nonnegative integers, respectively.

\section{\bf  Nappi-Witten algebras and their affinizations }
\subsection{Nappi-Witten Lie algebras}
The {\it Nappi-Witten Lie algebra} $H_{4}$ is  a $4$-dimensional vector space
$$
H_{4}=\bC a\oplus \bC b\oplus\bC c\oplus\bC d
$$
equipped with the bracket relations
$$
[a,b]=c,\ \ [d,a]=a,\ \ [d,b]=-b,\ \ [c,d]=[c,a]=[c,b]=0.
$$
Let $(,)$ be a symmetric bilinear form on $H_{4}$ defined by
$$
(a,b)=1,\quad (c,d)=1,\quad\text{otherwise}, \ (\ ,\ )=0.
$$
It is straightforward to check  that $(\ ,\ )$ is a non-degenerate
$H_{4}$-invariant symmetric bilinear form on $H_{4}$. The Casimir
element is defined as
\begin{equation}
\Omega=ab+ba+cd+dc\in U(H_{4}),\label{Ca}
\end{equation}
where $U(H_{4})$ is the universal enveloping algebra of $H_{4}$.
For any $0\neq\ell\in\bC$, we introduce the following modified Casimir element :
\begin{eqnarray}
\tilde\Omega_{\ell}=\Om-\frac{1}{\ell} c^2\in U(H_{4}).
\end{eqnarray}
It is clear that both $\Om$ and $\tilde\Om_\ell$ lie in the center of $U(H_{4})$.
\subsection{$H_4$-modules}
The Nappi-Witten algebra $H_4$ is equipped with the triangular decomposition:
$$
H_4=H_4^+\oplus H_4^0\oplus H_4^-=H_4^{\geq0}\oplus H_4^{-},
$$
where
$$
H_4^+=\bC a,\quad H_4^0=\bC c\oplus\bC d,\quad H_4^-=\bC b.
$$
For $\la\in (H_4^0)^*$,  the  highest weight module (Verma module) of  $H_4$ is defined by
\begin{equation}
M(\la)=U(H_4)\ot_{U(H_4^{\geq0})}\bC_{\la},\label{Verma}
\end{equation}
where $\bC_{\la}$ is the $1$-dimensional $H_4^{\geq0}$-module, on
which $h\in H_4^0$ acts as multiplication by $\la(h)$, and $H_4^+$
acts as $0$. For convenience, we denote $M(\la)$ by $M(\cc,\d)$ when
$\la(c)=\cc$ and $\la(d)=\d$.

The following lemma is well-known:

\begin{lem}[cf.\cite{MP}]
For $\cc,\d\in\bC$, $M(\cc,\d)$ is irreducible if and only if
$\cc\neq0$. If $\cc=0$, then the irreducible quotient
$L(\d)=L(0,\d)=\bC v_{\d}$ such that
\begin{equation}
av_{\d}=bv_\d=cv_{\d}=0,\quad dv_{\d}=\d v_{\d}.\label{Trivial}
\end{equation}
\end{lem}
\qed

Let $M=V(\al,\be,\ga)$ for $\al,\be,\ga\in\bC$, where
$V(\al,\be,\ga)$ is the $H_4$-module defined as follows:
 $V(\al,\be,\ga)=\bigoplus_{n\in\bZ}
\mathbb{C}v_{n}$ such that
\begin{equation}
dv_{n}=(\al+n)v_{n},\ cv_{n}=\be v_{n},\ av_{n}=-\be v_{n+1},\
bv_{n}=(\al+\ga+n)v_{n-1}\label{Im}.
\end{equation}
\begin{rem}
M. Willard \cite{Wi} showed that all irreducible weight modules of
$H_{4}$ with  finite-dimensional weight spaces can be classified
into the following classes:
\begin{itemize}
\item[(1)] Irreducible modules $L(\d)$ for $\d\in \bC$.
\item[(2)] Irreducible highest weight modules $M(\la)$ or irreducible lowest weight
modules $M^-(\mu)$.
\item[(3)] Intermediate series modules $V(\al,\be,\ga)$ defined as
(\ref{Im}) such that $\beta\neq 0$ and $\al+\gamma\notin \Z$.

\end{itemize}
\end{rem}

\subsection{Affine Nappi-Witten algebras }

To the pair $(H_{4},(\ ,\ ))$, we associate the {\it affine Nappi-Witten Lie algebra} $\widehat{H}_{4}$,
with the underlying vector space
\begin{equation}
\widehat{H}_{4}=H_{4}\ot\bC[t,t^{-1}]\oplus\bC{\bf k}
\end{equation}
equipped with the bracket relations
\begin{eqnarray}
[h_1\ot t^m,h_2\ot t^n]=[h_1,h_2]\ot
t^{m+n}+m(h_1,h_2)\delta_{m+n,0}{\bf k}, \ [\widehat{H}_{4}, {\bf
k}]=0, \label{DefinitonANW}
\end{eqnarray}
for $h_1,h_2\in H_{4}$ and $m,n\in\bZ$. It is clear that $\H$ has
the $\bZ$-grading:
$$
\H=\coprod_{n\in\bZ}\H^{(n)},
$$
where
$$
\H^{(0)}=H_{4}\oplus\bC\k,\quad\text{and}\quad\H^{(n)}=H_{4}\ot t^{n}\quad\text{for}\ n\neq0.
$$
Then one has the following graded subalgebras of $\H$:
\begin{eqnarray*}
\H^{(\pm)}=\coprod_{n>0}\H^{(\pm n)},\quad
\H^{(\geq0)}=\coprod_{n\geq0}\H^{(n)}= \H^{(+)}\oplus H_{4}\oplus\bC
\k.
\end{eqnarray*}

Let $M$ be an $H_{4}$-module and regard $M$ as an
$\H^{(\geq0)}$-module with $\H^{(+)}$ acting trivially and with $\k$
acting as the scalar $\ell$. The induced $\H$-module $V_{\H}(\ell, M)$ is
defined by
\begin{equation}
V_{\H}(\ell,M)=\Ind_{\H^{(\geq0)}}^{\H}(M)=U(\H)\ot_{U(\H^{(\geq0)})}M.\label{vermaH4}
\end{equation}

\section{Structure of (generlized) Verma modules}
\subsection{Verma modules}
Let $M=L(\d)=\bC v_{\d}$ for
$\d\in\bC$ be the one-dimensional irreducible $H_4$-module defined as in (\ref{Trivial}).
Next we shall investigate the structure of the following standard Verma module:
\begin{equation}
V_{\H}(\ell,\mathbbm{d})=U(\H)\ot_{U(\H^{(\geq0)})}\bC
v_{\d}.\label{Vac}
\end{equation}

Let us recall some basic concepts.
A partition $\la$ is a sequence $\la=(\la_1,\la_2,\dots ,\la_r)$ of
positive integers in decreasing order: $\la_1\geq\la_2\geq
\cdots\geq\la_r>0$. Denote $\P$ for the set of all partitions. We
call $r$ the length of $\la$,  denoted by $l(\la)$, and
 call the
sum of $\la_i's$  the weight of $\la$, denoted by $|\la|$.  Let
$\la=(\la_1,\la_2,\dots ,\la_r)$, $\mu=(\mu_1,\mu_2,\dots ,\mu_s)$
be two partitions. If $r<s$, we rewrite $\la$ as
$\la=(\la_1,\la_2,\dots ,\la_r, \la_{r+1}, \dots, \la_{s})$, where
$\la_{r+1}=\cdots=\la_{s}=0$. If $r>s$, we rewrite $\mu$ in a
similar way.  The natural ordering on partitions is defined as
follows:
 $$\la>\mu \iff
\la_1= \mu_{1}, \cdots, \la_i=\mu_i, \ \la_{i+1}>\mu_{i+1}, \ {\rm
for \ some} \  i\geq 0.$$
$$\la=\mu \iff \la_i=\mu_i, \ {\rm
for \ all} \  i\geq1.$$ For $k\geq 1$ and  partitions
$\la^{1},\cdots,\la^{k}, \mu^{1},\cdots, \mu^{k}$, we define
$$(\la^{1},\cdots,\la^{k})\succ (\mu^{1},\cdots,
\mu^{k})\Longleftrightarrow
\la^{1}=\mu^{1},\cdots,\la^{i}=\mu^{i},\la^{i+1}>\mu^{i+1}, \ {\rm
for \ some} \ 0\leq i\leq k.$$ For convenience, we shall denote  $x\ot t^{n}$ by $x(n)$ for $x \in
H_{4}, n\in\bZ$. For $x\in H_4$, $\la=(\la_1,\la_2,\dots ,\la_r)\in \mathcal{P}$,
denote
$$
x(\la):=x(\la_r)\cdots x(\la_1),\quad x(-\la):=x(-\la_1)\cdots x(-\la_r).
$$
Let
\begin{eqnarray*}
&&S_{2}=\{(d,a,b), (d,c, a), (d,c, b), (c,a,b)\},\\
&&S_{3}=\{(d,a),(d,b),(d,c),(a,b),(c,a),(c,b)\},\\
&&S_{4}=\{d,c,a,b\};
\end{eqnarray*}
and
\begin{eqnarray*}
&&\mathcal{B}_{1}^{\pm}=\{d(\pm\la)c(\pm\mu)a(\pm\nu)b(\pm\eta),\,|\,\la,\mu,\nu,
\eta\in \mathcal{P}\},\\
&&\mathcal{B}_{2}^{\pm}=\{x_{1}(\pm\la)x_{2}(\pm\mu)x_{3}(\pm\nu)\,|\,(x_{1},x_{2},x_{3})\in
S_{2}, \la,\mu,\nu\in \mathcal{P}\},\\
&&\mathcal{B}_{3}^{\pm}=\{x_{1}(\pm\la)x_{2}(\pm\mu)\,|\,(x_{1},x_{2})\in
S_{3}, \la,\mu\in \mathcal{P}\},\\
&&\mathcal{B}_{4}^{\pm}=\{x(\pm\la)\,|\,x\in S_{4},\, \la\in
\mathcal{P}\}.
\end{eqnarray*}
Then
$$
\mathcal{B}^{\pm}=\bigcup_{i=1}^{4}\mathcal{B}_{i}^{\pm}\label{PBW}
$$
is a PBW basis for $U(\H^{(\pm)})$. For
$X=d(-\la)c(-\mu)a(-\nu)b(-\eta)\in\B_{1}^{-}$,
$Y=d(\la)c(\mu)a(\nu)b(\eta)\in\B_{1}^{+}$, we define
$$
\Ht(X)=|\la|+|\mu|+|\nu|+|\eta|, \ \
\Ht(Y)=-(|\la|+|\mu|+|\nu|+|\eta|).
$$
Similarly,  we can define $\Ht(Z)$ for other
elements $Z$ in $\mathcal{B^{\pm}}$. The following lemma is obvious.
\begin{lem}\label{adl}  Let $X\in \B^{-}$, $Y\in \B^{+}$ be such that
$\Ht(X)<-\Ht(Y)$. Then $Y(Xv_{\d})=0$.
\end{lem}

We are now in a position to give  the first main result of this
section.
\begin{theo}\label{Theorem1}
For $\ell,\mathbbm{d}\in\bC$, the $\H$-module
$V_{\H}(\ell,\mathbbm{d})$ is  irreducible  if and only if
$\ell\neq0$. Furthermore, if $\ell=0$, then the irreducible quotient of
$V_{\H}(\ell,\mathbbm{d})$ is isomorphic to the one-dimensional $\H$-module $\C v_{\d}$.
\end{theo}
\begin{proof}

 For $v=Xv_{\d}\in V_{\H}(\ell,\mathbbm{d})$, $X\in
\mathcal{B}$, we denote
$$ \deg v=\Ht(X),\quad \deg v_{\d}=0.$$
Then $V_{\H}(\ell,\mathbbm{d})$ is ${\mathbb N}$-graded.

 If $\ell=0$, it is easy to see
that $\{x(-1)v_{\d}\ | \ x=a,b,c,d\}$ are  singular vectors and
generate the maximal non-trivial submodule of
$V_{\H}(\ell,\mathbbm{d})$, which implies that $V_{\H}(\ell,\d)$ is
reducible and its irreducible quotient is  $\C v_{\d}$.

Conversely, let $\ell\neq 0$. By PBW theorem and the definition of
$V_{\H}(\ell,\mathbbm{d})$, we have
$$\{v_{\d}, xv_{\d} \ | \ x\in \B^{-}\}$$ is a basis of
$V_{\H}(\ell,\mathbbm{d})$. For $n\in{\mathbb Z}_{+}$, let
$$p_{n}=c(n), \ \ q_{n}=d(-n).$$
Then
$$[p_{m},q_{n}]=m\delta_{m,n}\k.$$
Therefore $${\frak s}=\oplus_{n\in{\mathbb Z}_{+}}(\C p_{n}\oplus\C
q_{n})\oplus\C\k$$ is a Heisenberg algebra and
$V_{\H}(\ell,\mathbbm{d})$ is an ${\frak s}$-module such that  $\k$
acts as $\ell\neq 0$. Since every highest weight ${\frak s}$-module
generated by one element with $\k$ acting as a non-zero scalar is
irreducible, it follows  that $V_{\H}(\ell,\mathbbm{d})$ can be decomposed into a direct
sum of irreducible highest weight modules of ${\frak s}$ with
the highest weight vectors in
$${\mathcal N}=\{v_{\d},
c(-\la)a(-\mu)b(-\nu)v_{\d}, y_{1}(-\la)y_{2}(-\mu)v_{\d},
x(-\la)v_{\d} \
$$$$ | \ (y_{1},y_{2})\in\{(a,b),(c,a),(c,b)\}, x\in\{a,b,c\}, \la,\mu,\nu\in{\mathcal P}\}.$$
Now suppose that, for $\ell\neq0$,  $V_{\H}(\ell,\mathbbm{d})$ is not irreducible and let $U$ be
a  proper non-zero submodule of $V_{\H}(\ell,\mathbbm{d})$. Then there exists a
non-zero homogeneous singular vector $u$ in $U$ such that $u$ is a linear combination of  elements
in ${\mathcal N}$  such that
$$Xu=0,\quad\text{for all}\   X\in U(\H^{(+)}).$$
We will show that there exists an
element $Y$ in $\B^{+}$ such that $Yu\neq 0$, giving a
contradiction. Hence we can obtain $V_{\H}(\ell,\mathbbm{d})$ is  irreducible. We may
assume that
$$u=\sum_{i=1}^{3}X_{i}v_{\d},
$$where
$\Ht(X_{1})=\Ht(X_{2})=\Ht(X_{3})$ and
\begin{eqnarray*}
&&X_{1}=\sum_{j=1}^{l_{1}}a_{1j}c(-\la^{(1j)})a(-\mu^{(1j)})b(-\nu^{(1j)}),\\
\end{eqnarray*}
where $
(\la^{(1j)},\mu^{(1j)},\nu^{(1j)})\succ(\la^{(1,j+1)},\mu^{(1,j+1)},\nu^{(1,j+1)})$
for $ j=1,\cdots, l_{1}-1$,
\begin{eqnarray*}
X_{2}&=&X_{21}+X_{22}+X_{23},
\end{eqnarray*}
where
\begin{eqnarray*}
&&X_{21}=\sum_{j=1}^{l_{2}}a_{2j}c(-\la^{(2j)})a(-\mu^{(2j)}),\quad X_{22}=\sum_{j=1}^{l_{3}}a_{3j}c(-\la^{(3j)})b(-\nu^{(3j)}),\\
&&X_{23}=\sum_{j=1}^{l_{4}}a_{4j}a(-\mu^{(4j)})b(-\nu^{(4j)}),
\end{eqnarray*}
for $(\la^{(2j)},\mu^{(2j)})\succ(\la^{(2,j+1)},\mu^{(2,j+1)})$
$j=1,\cdots, l_{2}-1; (\la^{(3j)},\nu^{(3j)})\succ(\la^{(3,j+1)},\mu^{(3,j+1)})$ $j=1,\cdots, l_{3}-1$; and
$(\mu^{(4j)},\nu^{(4j)})\succ(\mu^{(4,j+1)},\nu^{(4,j+1)})$ $j=1,\cdots, l_{4}-1,$
and
\begin{eqnarray*}
X_{3}&=&X_{31}+X_{32}+X_{33},
\end{eqnarray*}
where
$$
X_{31}=\sum_{j=1}^{l_{5}}a_{5j}c(-\la^{(5j)}),\ X_{32}=\sum_{j=1}^{l_{6}}a_{6j}a(-\mu^{(6j)}),\
X_{33}=\sum_{j=1}^{l_{7}}a_{7j}b(-\nu^{(7j)}),
$$
for $\la^{(5j)}\succ\la^{(5,j+1)}$ $ j=1,\cdots,l_{5}-1$;
$ \mu^{(6j)}\succ\mu^{(6,j+1)}$ $j=1,\cdots,l_{6}-1;$ and
$ \nu^{(7j)}\succ\nu^{(7,j+1)}$ $ j=1,\cdots,l_{7}-1$.

We break up the proof into seven different cases.
\begin{itemize}
\item[(1)]$X_{32}\neq 0$.  Let
$Y=b(\mu^{(61)})$, then
$$Yu=a_{61}l^{r_{61}}\prod_{j=1}^{r_{61}}\mu_{j}^{(61)}v_{\d}\neq 0,$$
where $\mu^{(61)}=(\mu_{1}^{(61)}, \cdots, \mu_{r_{61}}^{(61)})$.
\item[(2)] $X_{32}=0$, $X_{33}\neq 0$. Taking
$Y=a(\nu^{(71)})$,  we have
$$YX_{1}v_{\d}=YX_{2}v_{\d}=0,\quad YX_{3}v_{\d}\neq 0.$$
Hence $Yu\neq 0$.
\item[(3)] $X_{32}=X_{33}=0$, $X_{23}\neq 0$.
Let $Y=a(\nu^{(41)})b(\mu^{(41)})$, then
$$
YX_{1}v_{\d}=YX_{3}v_{\d}=0,\quad YX_{2}v_{\d}\neq 0.
$$
which implies that $Yu\neq 0$.
\item[(4)]$X_{23}=X_{32}=X_{33}=0$, $X_{31}\neq 0$. Let
$Y=d(\la^{(51)})$, then we have
$$
YX_{1}v_{\d}=YX_{2}v_{\d}=0,\quad YX_{3}v_{\d}\neq 0.
$$
Hence $Yu\neq 0$.
\item[(5)] $X_{3}=X_{23}=0$, $X_{1}\neq 0$. Let
$Y=a(\nu^{(11)})b(\mu^{(11)})d(\la^{(11)})$, then
$$
YX_{2}v_{\d}=0,\quad YX_{1}v_{\d}\neq 0.$$
Hence $Yu\neq 0$.
\item[(6)]$X_{1}=X_{3}=0$, $X_{23}=0$, $X_{21}\neq 0$.
Let $Y=b(\mu^{(21)})d(\la^{(21)})$, then $Yu\neq 0$.
\item[(7)]$X_{1}=X_{3}=0$, $X_{21}=X_{23}=0$, $X_{22}\neq 0$.
Let $Y=a(\nu^{(31)})d(\la^{(31)})$, then $Yu\neq 0$.
\end{itemize}
 We complete the
proof of the theorem.

\end{proof}

\subsection{Generalized Verma modules}
Next we
shall consider the  generalized Verma modules
\begin{eqnarray*}
V_{\H}(\ell,\mathbbm{c},\d)=V_{\H}(\ell,M(\la))=\Ind_{\H^{(\geq
0)}}^{\H}M(\la)
\end{eqnarray*}
as defined in (\ref{vermaH4}), where  $M(\la)$  is the $H_4$-module
defined  in (\ref{Verma}) such that $\la(c)=\mathbbm{c}\in\bC^*$ and
$\la(d)=\d\in\bC$. It is clear that $\H$  has the following new triangular
 decomposition:
\begin{eqnarray}
\H= \H^{-}\oplus  \H^{0}\oplus  \H^{+},\label{NT}
\end{eqnarray}
where
$$
\H^{+}=\mathbb{C}a\oplus \H^{(>0)},\quad \H^{-}=\mathbb{C}b\oplus
\H^{(<0)},\quad \H^0=\mathbb{C}c\oplus\mathbb{C}d\oplus \mathbb{C}\k.
$$
According to  this new triangular decomposition,  for $\d,\ell\in\bC$ and $\cc\in\bC^*$, we have the standard Verma module
\begin{eqnarray}
V^{new}_{\H}(\la)=U(\H)\otimes_{U(\widehat{H}_{4}^{+}+\H^0)}w_{\la}.
\end{eqnarray}
where $\la\in(\H^0)^*$, $\H^+w_{\la}=0$ and $xw_{\la}=\la(x)w_{\la}$ for all $x\in\H^0$.
Furthermore, if
$$
\la(c)=\cc ,\quad \la(d)=\d ,\quad \la(\k)=\ell ,
$$
then the generalized Verma module $V_{\H}(\ell,\cc,\d)$ is
isomorphic to the standard Verma module $V^{new}_{\H}(\la)$ as
$\H$-modules.

The second main result of this section is stated as follows:

\begin{theo}For $\ell,\d\in\bC$ and $\cc\in\bC^*$,  the
$\H$-module $V_{\H}(\ell,\mathbbm{c},\d)$ is irreducible if and only
if $\ell\neq 0$ and $\cc\notin \ell\Z$. Furthermore,  we have

\begin{itemize}
\item[(i)] If $\ell\neq 0$ and $\cc+m\ell=0$, for some $m\in\Z_{+}$, then all the
linearly independent singular vectors are
$$
u=[\sum_{\la\in
P(m)}(a_{\la}c(-\la)b+\sum_{i=1}^{k_{\la}}b_{\la\setminus\la_{i},\la_{i}}c(-\la\setminus\la_{i})b(-\la_{i}))]^{k}v_{\d},
\ $$ for all $k\in\Z_{+}$,  satisfying
\begin{eqnarray*}
&&a_{\la}q_{i}\la_{i}\ell-b_{\la\setminus\la_{i},\la_{i}}=0, \ i=1,2,\cdots,k_{\la},\\
&&a_{\la}\cc+\sum_{i=1}^{k_{\la}}b_{\la\setminus\la_{i},\la_{i}}=0,\\
&&b_{\la\setminus\la_{i},\la_{i}}(\cc+\la_{i}\ell)
+\sum_{j=1}^{k_{\la}}b_{{\la}\setminus\la_{i}\setminus\la_{j},\la_{i}+\la_{j}}=0;
\end{eqnarray*}
\item[(ii)]If $\ell\neq 0$ and $\cc-m\ell=0$, for some $m\in\Z_{+}$, then all the
linearly independent singular vectors are
$$
u=[\sum_{\la\in
P(m)}\sum_{i=1}^{k_{\la}}c_{\la\setminus\la_{i},\la_{i}}c(-\la\setminus\la_{i})a(-\la_{i}))]^{k}v_{\d},$$
for all $k\in\Z_{+}$,  satisfying
$$
q_{i}\la_{i}\ell c_{\la\setminus\la_{j},\la_{j}}+c_{\la\setminus\la_{i}\setminus\la_{j},\la_{i}+\la_{j}}=0,
\ i,j=1,2,\cdots,k_{\la}, i\neq j,
$$$$\label{dd}c_{\la\setminus\la_{i},\la_{i}}(-\cc+\la_{i}\ell)-\sum_{j=1}^{k_{\la}}c_{\la\setminus\la_{i}\setminus\la_{j},\la_{i}+\la_{j}}=0,
\ i=1,2,\cdots,k_{\la},
$$
where in (i) and (ii), $P(m)$ is the set of all partitions of weight
$m$,
$$
\la^{(k)}=(\la,\la,\cdots,\la)\in \Z_{+}^{k},\quad
\la=(\la_{1}^{(q_{1})}, \la_{2}^{(q_{2})},\cdots,
\la_{k_{\la}}^{(q_{k_{\la}})})$$ such that
$\sum_{i=1}^{k_{\la}}q_{i}\la_{i}=m$, and
$\la\setminus\la_{i}=(\la_{1}^{(q_{1})},\cdots,
\la_{i}^{(q_{i}-1)},\cdots, \la_{k_{\la}}^{(q_{k_{\la}})})$ and if
$\la=\la_{i}$, then
$c(-\la\setminus\la_{i})b(-\la_{i})=b(-\la_{i})$,
$c(-\la\setminus\la_{i})a(-\la_{i})=a(-\la_{i})$;
\item[(iii)] If $\ell=0$, then all the linearly independent singular vectors are
$$u=c(-\la)v_{\d}, $$
for all partitions $\la\in{\mathcal P}$.
\end{itemize}

\end{theo}
\begin{proof} If  $\ell=0$, then it is easy to see that $\{c(-\la)v_{\d} \ | \ \la\in{\mathcal P}\}$
 are all the linearly independent singular vectors.

Now assume that $\ell\neq 0$.

Let ${\frak s}$ be the infinite-dimensional Heisenberg algebra
defined in the proof of Theorem \ref{Theorem1}. Then as an ${\frak s}$-module,
$V_{\H}(\ell,\cc,\d)$ is a direct sum of irreducible highest weight
${\frak s}$-modules with $\k$ acting as $\ell$ and highest weight
vectors in
$$
{\mathcal N}_{\la}=\{b^{i}v_{\d}, c(-\la)a(-\mu)b(-\nu)b^{i}v_{\d},
y_{1}(-\la)y_{2}(-\mu)b^{i}v_{\d},$$$$ x(-\la)b^{i}v_{\d} \ | \
i\geq 0, (y_{1},y_{2})\in\{(a,b),(c,a),(c,b)\}, x\in\{a,b,c\},
\la,\mu,\nu\in{\mathcal P}\}.$$ Let
$$\widehat{H}_4^{+}=\widehat{H}_4^{(+)}+H_{4}^{+}, $$
where $\widehat{H}_4^{(+)}$ and $H_{4}^{+}$ are defined as in Section
2.3. For $n\in\Z$, define
\begin{eqnarray*}
&&\deg a(-n)=n-1,\ \deg b(-n)=n+1,\deg c(-n)=n,\  \deg d(-n)=n,\quad\deg\k=0.
\end{eqnarray*}
Then $U(\widehat{H}_4)$ is $\Z$-graded. So we get an $\N$-gradation
of $V^{new}_{\H}(\ell,\cc,\d)$ by defining $\deg v_{\d}=0$.

Suppose that $u\in V_{\H}(\ell,\cc,\d)$ is a singular vector such
that
$$U(\widehat{H}_4^{+})u=0.$$ Since $V_{\H}(\ell,\cc,\d)$ is a direct sum of irreducible highest
weight ${\frak s}$-modules  and $M(\la)$ is irreducible, we may
assume that $u$ is homogeneous and
$$
u=\sum_{i=1}^{3}u_{i},
$$
where
$$u_{1}=\sum_{j=1}^{l_{1}}a_{1j}c(-\la^{(1j)})a(-\mu^{(1j)})b(-\nu^{(1j)})b^{k_{1j}}v_{\d},$$
for $(\la^{(1j)},\mu^{(1j)},\nu^{(1j)})\succ(\la^{(1,j+1)},\mu^{(1,j+1)},\nu^{(1,j+1)})
, j=1,\cdots, l_{1}-1$,
$$
u_{2}=u_{21}+u_{22}+u_{23},$$
where
$$u_{21}=\sum_{j=1}^{l_{2}}a_{2j}c(-\la^{(2j)})a(-\mu^{(2j)})b^{k_{2j}}v_{\d},
\ u_{22}=\sum_{j=1}^{l_{3}}a_{3j}c(-\la^{(3j)})b(-\nu^{(3j)})b^{k_{3j}}v_{\d},
$$
$$
u_{23}=\sum_{j=1}^{l_{4}}a_{4j}a(-\mu^{(4j)})b(-\nu^{(4j)})b^{k_{4j}}v_{\d},
$$
for $(\la^{(2j)},\mu^{(2j)})\succ(\la^{(2,j+1)},\mu^{(2,j+1)}),
j=1,\cdots, l_{2}-1$;
$(\la^{(3j)},\nu^{(3j)})\succ(\la^{(3,j+1)},\mu^{(3,j+1)}),
j=1,\cdots, l_{3}-1$;
$(\mu^{(4j)},\nu^{(4j)})\succ(\mu^{(4,j+1)},\nu^{(4,j+1)}),
j=1,\cdots, l_{4}-1$, and
$$
u_{3}=u_{31}+u_{32}+u_{33},$$
where
$$
u_{31}=\sum_{j=1}^{l_{5}}a_{5j}c(-\la^{(5j)})b^{k_{5j}}v_{\d}, $$$$
u_{32}=\sum_{j=1}^{l_{6}}a_{6j}a(-\mu^{(6j)})b^{k_{6j}}v_{\d},\
u_{33}=\sum_{j=1}^{l_{7}}a_{7j}b(-\nu^{(7j)})b^{k_{7j}}v_{\d},$$
for $\la^{(5j)}\succ\la^{(5,j+1)},
j=1,\cdots,l_{5}-1$; $\mu^{(6j)}\succ\mu^{(6,j+1)},
j=1,\cdots,l_{6}-1$; $\nu^{(7j)}\succ\nu^{(7,j+1)},
j=1,\cdots,l_{7}-1$.

{\bf Case 1} \  $\ell\neq 0$ and $\cc\notin \ell\Z$. One can prove that
$u=0$ by the same method used in the proof of Theorem
\ref{Theorem1}. Therefore $V^{new}_{\H}(\ell,\cc,\d)$ is
irreducible.

{\bf Case 2} \ $\ell\neq 0$ and  $\cc\in \ell\Z$.

Since $au=0$, it is easy to see that $k_{ij}=0$, for $i=4,6,7$. If
$u_{33}\neq 0$, we consider $a(\nu_{1}^{(71)})u$. Since no monomial
of $a(\nu_{1}^{(71)})u_{1}$, $a(\nu_{1}^{(71)})u_{2}$,
$a(\nu_{1}^{(71)})u_{31}$ and $a(\nu_{1}^{(71)})u_{32}$ is of the
form $b(-\eta)$, where $\eta\in{\mathcal P}$, we deduce that
$$
\cc+\nu_{1}^{(71)}\ell=0.$$ Similarly, considering
$a(\nu_{r_{7p}}^{(7p)})u$, where
$\nu_{r_{7p}}^{(7p)}=min\{\nu_{r_{7j}}^{(7j)}, \
j=1,2,\cdots,l_{7}\}$, we have
$$
\cc+\nu_{r_{7p}}^{(7p)}\ell=0.$$ Then we deduce that $l_{7}=1$ and
$u_{33}=a_{31}(b(-\nu_{1}^{(71)}))^{k_{3}}v_{\d}$ such that
$$\cc+\nu_{1}^{(71)}\ell=0, \ \ k_{3}(\nu_{1}^{(71)}+1)=\deg u.$$
If $u_{32}\neq 0$, then by the fact that $b(\mu_{1}^{(61)})u=0$, we
have
$$
\cc-\mu_{1}^{(61)}\ell=0.$$ Therefore
$$
(\mu_{1}^{(61)}+\nu_{1}^{(71)})\ell=0,$$ which is impossible since
$\mu_{1}^{(61)}+\nu_{1}^{(71)}>0$ and $\ell\neq 0$. We deduce that
$u_{32}=0$ or $u_{33}=0$.

{\bf Subcase 1} \ $u_{33}\neq 0$, $u_{32}=0$.

By the fact that
$b(\mu_{1}^{(41)})u=b(\mu_{1}^{(41)})u_{31}=b(\mu_{1}^{(41)})u_{33}=b(\mu_{1}^{(41)})u_{22}=0$,
and both  $b(\mu_{1}^{(41)})u_{1}$ and $b(\mu_{1}^{(41)})u_{21}$
contain no monomials of the forms $a(-\la)b(-\mu)$ and $b(-\mu)$,
where $\la,\mu\in{\mathcal P}$, we have
$$
b(\mu_{1}^{(41)})u_{23}=\sum_{j=1}^{l_{4}}a_{4j}\sum_{p=1}^{r_{4j}}
\delta_{\mu_{1}^{(41)},\mu_{p}^{(4j)}}(-\cc+\mu_{1}^{(41)}\ell)$$
$$\cdot a(-\mu^{(4j)})\widehat{a(-\mu_{p}^{(4j)})}b(-\nu^{(4j)})v_{d}=0,$$
where $\widehat{a(-\mu_{p}^{(4j)})}$ means this factor is deleted.
 So if
$u_{23}\neq 0$, then $-\cc +\mu_{1}^{(41)}\ell=0$, which is impossible
since $\cc+\nu_{1}^{(71)}\ell=0.$ This proves that $u_{23}=0.$ Let
$Y=b(\mu_{1}^{(21)})$, then $Yu=Yu_{1}+Yu_{21}=0$. If $u_{21}\neq
0$, comparing $Yu_{1}$ and $Yu_{21}$, we have
$$
-\cc +\mu_{1}^{(21)}\ell=0,$$ which is not true. So $u_{21}=0$.
Similarly, $u_{1}=0.$ Therefore
$$u=u_{22}+u_{31}+u_{33}.$$
 By the fact that
$a(\nu_{1}^{(31)})u=0$, we can easily deduce that $\nu^{(3j)}\preceq
\nu^{(71)}$, $j=1,\cdots,l_{3}$. Similarly, we have
$\la^{(3j)}\preceq \nu^{(71)},\la^{(5k)}\preceq \nu^{(71)}$,
$j=1,\cdots,l_{3}, k=1,\cdots, l_{5}$. If $\nu_{1}^{(71)}=1$, then
$\cc+\ell=0$, $u_{33}=a_{71}b(-1)^{k_{3}}v_{\d}$ and
$$u_{22}=\sum_{j=1}^{l_{3}}a_{3j}c(-1)^{p_{3j}}b(-1)^{q_{3j}}b^{k_{3j}}v_{\d},\
u_{31}=\sum_{j=1}^{l_{5}}a_{5j}c(-1)^{p_{5j}}b^{k_{5j}}v_{\d}.$$
Since $d(1)u=0$, we have
$$
u_{22}+u_{31}+u_{33}=k(\sum_{j=0}^{k_{3}}\ell^{-k_{3}+j}C_{k_{3}}^{j}c(-1)^{k_{3}-j}b(-1)^{j}b^{k_{3}-j}v_{\d}),$$
for some $k\in \C$. Let $$
u=\sum_{j=0}^{k_{3}}\ell^{-k_{3}+j}C_{k_{3}}^{j}c(-1)^{k_{3}-j}b(-1)^{j}b^{k_{3}-j}v_{\d}=(\ell^{-1}c(-1)b+b(-1))^{k_{3}}v_{\d}.$$
Then we have
$$
au=\sum_{j=0}^{k_{3}}\ell^{-k_{3}+j}C_{k_{3}}^{j}(jc(-1)^{k_{3}-j+1}b(-1)^{j-1}b^{k_{3}-j}+(k_{3}-j)\cc
c(-1)^{k_{3}-j} b(-1)^{j}b^{k_{3}-j-1})v_{\d}=0.$$ It is clear that
$a(n)u=b(n)u=c(n)u=d(m)u=0$, for $n\geq 1, m\geq 2$. We prove that
$u$ is a non-zero singular vector for each $k_{3}\in \N$.

Assume that $\nu_{1}^{(71)}>1$. If $\nu_{1}^{(3j)}=1$ for all
$j=1,2,\cdots,l_{3}$. Then $\nu_{1}^{(71)}=2$. In fact, if
$\nu_{1}^{(71)}>2$, then
$a(\nu_{1}^{(71)}-1)u_{22}=a(\nu_{1}^{(71)}-1)u_{31}=0$,
$a(\nu_{1}^{(71)}-1)u_{33}\neq 0$, a contradiction. If $k_{3}>1$,
then $a(1)u_{31}=0$, and $a(1)u_{33}$ contains factor $b(-2)$, but
no monomial of $a(1)u_{22}$ contains factor $b(-2)$. So
$u_{33}=u_{22}=0.$ Then $u_{31}=0$ and we deduce that $u=0$. If
$k_{3}=1$, then $u_{33}=a_{71}b(-2)v_{\d}$ and one can easily deduce
that
$$
u=(\ell c(-2)b+c(-1)^{2}b+2\ell c(-1)b(-1)+2\ell^{2}b(-2))v_{\d}$$ is a
singular vector. Generally, for $\nu_{1}^{(71)}=2$, we have
$$u=(\ell c(-2)b+c(-1)^{2}b-\cc c(-1)b(-1)-\ell\cc b(-2))^{k}v_{\d}, \ \ k\geq
1.$$

Now assume that $\nu_{1}^{(71)}=m>2$ and $l(\nu^{(71)})=1$. Let
$1\leq p\leq l_{3}$ be  such that $l(\nu^{(3p)})=max\{l(\nu^{(3j)}),
j=1,2,\cdots,l_{3}\} $ and if $l(\nu^{(3q)})=l(\nu^{(3p)})$, then
$\nu^{(3p)}\succeq \nu^{(3q)}$. If $l(\nu^{(3p)})>1$, then
$d(\la^{(3p)})u\neq 0$, a contradiction. So $l(\nu^{(3j)})=1$ for
all $j=1,2,\cdots,l_{3}$. It follows that $k_{3j}=0,
j=1,\cdots,l_{3}$ since $au=0$.  For $x\in\Z_{+}, k\in\Z_{+}$,
denote $(x,\cdots,x)\in\N^{k}$ by $x^{(k)}$. Then
$$u=\sum_{\la\in
P(m)}[a_{\la}c(-\la)b+\sum_{i=1}^{k_{\la}}b_{\la\setminus\la_{i},\la_{i}}c(-\la\setminus\la_{i})b(-\la_{i})]v_{\d},$$
where $P(m)$ is the set of all partitions of weight $m$,
$\la=(\la_{1}^{(q_{1})}, \la_{2}^{(q_{2})},\cdots,
\la_{k_{\la}}^{(q_{k_{\la}})})$ such that
$\sum_{i=1}^{k_{\la}}q_{i}\la_{i}=m$, and
$\la\setminus\la_{i}=(\la_{1}^{(q_{1})},\cdots,
\la_{i}^{(q_{i}-1)},\cdots, \la_{k_{\la}}^{(q_{k_{\la}})})$. By the
fact that $d(n)u=a(m)u=0$, for $m\in\N$, $n\in\Z_{+}$, and
$\cc+m\ell=0$, we deduce  that $a_{\la}$,
$b_{\la\setminus\la_{i},\la_{i}}, i=1,2,\cdots,k_{\la}$ are uniquely
determined by the following equations up to a non-zero scalar.
\begin{equation}
\label{aa} a_{\la}q_{i}\la_{i}\ell-b_{\la\setminus\la_{i},\la_{i}}=0, \
i=1,2,\cdots,k_{\la}
\end{equation}
\begin{equation}\label{bb}a_{\la}\cc+\sum_{i=1}^{k_{\la}}b_{\la\setminus\la_{i},\la_{i}}=0,
\ b_{\la\setminus\la_{i},\la_{i}}(\cc+\la_{i}\ell)
+\sum_{j=1}^{k_{\la}}b_{{\la}\setminus\la_{i}\setminus\la_{j},\la_{i}+\la_{j}}=0.
\end{equation}
It is easy to check that the $u$ determined by (\ref{aa}) and
(\ref{bb}) is indeed a non-zero singular vector. Generally
$$u=[\sum_{\la\in
P(m)}(a_{\la}c(-\la)b+\sum_{i=1}^{k_{\la}}b_{\la\setminus\la_{i},\la_{i}}c(-\la\setminus\la_{i})b(-\la_{i}))]^{k}v_{\d},
\ k=1,2,\cdots$$ satisfying (\ref{aa}) and (\ref{bb}) are all the
linearly independent singular vectors.

{\bf Subcase 2} $u_{33}=0$, $u_{32}\neq 0$. Then
$-\cc+\mu_{1}^{(61)}\ell=0$.  Similar to the proof for Subcase 1, we
can deduce that $u_{1}=u_{22}=u_{23}=0$, $l_{6}=1$,
$\mu^{(61)}=(\mu_{1}^{(61)},\cdots,\mu_{1}^{(61)})$ and
 $k_{2i}=k_{5j}=0,i=1,2,\cdots,l_{2},
j=1,2,\cdots,l_{5}$.

We first assume that $l(\mu^{(61)})=1$ and $\mu_{1}^{(61)}=m$. Then
it is easy to see that $l(\mu^{(2j)})=1$, $j=1,2,\cdots,l_{2}$,
$u_{21}=0$. Therefore
$$
u=\sum_{\la\in
P(m)}\sum_{i=1}^{k_{\la}}c_{\la\setminus\la_{i},\la_{i}}c(-\la\setminus\la_{i})a(-\la_{i}))v_{\d},
$$
where $P(m)$ is the set of all partitions of weight $m$,
$\la^{(k)}=(\la,\la,\cdots,\la)\in \Z_{+}^{k}$,
$\la=(\la_{1}^{(q_{1})}, $ $\la_{2}^{(q_{2})},\cdots,
\la_{k_{\la}}^{(q_{k_{\la}})})$ such that
$\sum_{i=1}^{k_{\la}}q_{i}\la_{i}=m$, and
$\la\setminus\la_{i}=(\la_{1}^{(q_{1})},\cdots,
\la_{i}^{(q_{i}-1)},\cdots, \la_{k_{\la}}^{(q_{k_{\la}})})$. By
$d(n)u=b(n)u=0$, for $n\in\Z_{+}$, we deduce that
\begin{equation}
\label{cc}
q_{i}\la_{i}\ell c_{\la\setminus\la_{j},\la_{j}}+c_{\la\setminus\la_{i}\setminus\la_{j},\la_{i}+\la_{j}}=0,
\ i,j=1,2,\cdots,k_{\la}, i\neq j,
\end{equation}
\begin{equation}\label{ee}c_{\la\setminus\la_{i},\la_{i}}(-\cc+\la_{i}\ell)-\sum_{j=1}^{k_{\la}}c_{\la\setminus\la_{i}\setminus\la_{j},\la_{i}+\la_{j}}=0,
\ i=1,2,\cdots,k_{\la}.
\end{equation}
Actually, $u$ is uniquely determined by (\ref{cc}) and (\ref{ee}).
Then one can easily deduce that all the linearly independent
singular vectors are
$$
u=[\sum_{\la\in
P(m)}\sum_{i=1}^{k_{\la}}c_{\la,\la_{i}}c(-\la\setminus\la_{i})a(-\la_{i}))]^{k}v_{\d},\
k\in\Z_{+}$$ satisfying (\ref{cc}) and (\ref{ee}).
\end{proof}

\subsection{}
In this subsection, we shall study the irreducibility of $\H$-module
$$
V_{\H}(\ell,\al,\be,\ga)=\Ind_{H_{4}^{(\geq
0)}}^{\H}V(\al,\be,\ga)$$ with $\k$ acting as a scalar $\ell$. It is
obvious that $V(\al,\be,\ga)$ is irreducible if and only if $\beta\neq 0$ and
$\al+\ga\notin \bZ$. Similarly, we have
\begin{theo}Let $\al,\be,\ga\in\bC$ be such that $\al+\ga\notin \bZ$, $\beta\neq 0$. Then  the $\H$-module
$V_{\H}(\ell,\al,\be,\ga)$ is irreducible if and only if
$\beta+n\ell\neq 0$ for all $n\in \mathbb{Z}$. Furthermore, we have
\begin{itemize}
\item[(i)]If $\ell\neq 0$ and $\be+ml=0$, for some $m\in\Z_{+}$, then for each
$k\in\Z_{+}$,
$$
u=[\sum_{\la\in
P(m)}(a_{\la}c(-\la)b+\sum_{i=1}^{k_{\la}}b_{\la\setminus\la_{i},\la_{i}}c(-\la\setminus\la_{i})b(-\la_{i}))]^{k}v_{0},
\ $$  satisfying
$$a_{\la}q_{i}\la_{i}\ell-b_{\la\setminus\la_{i},\la_{i}}=0, \ i=1,2,\cdots,k_{\la},
$$$$
a_{\la}\cc+\sum_{i=1}^{k_{\la}}b_{\la\setminus\la_{i},\la_{i}}=0, \
b_{\la\setminus\la_{i},\la_{i}}(\cc+\la_{i}l)+\sum_{j=1}^{k_{\la}}b_{{\la}\setminus\la_{i}\setminus\la_{j},\la_{i}+\la_{j}}=0$$
generates a non-trivial submodule of $V_{\H}(\ell,\al,\be,\ga)$;

\item[(ii)] If $\ell\neq 0$ and $\be-m\ell=0$, for some $m\in\Z_{+}$, then for
each $k\in\Z_{+}$,
$$
u=[\sum_{\la\in
P(m)}\sum_{i=1}^{k_{\la}}c_{\la\setminus\la_{i},\la_{i}}c(-\la\setminus\la_{i})a(-\la_{i}))]^{k}v_{0}$$
 satisfying
$$
q_{i}\la_{i}\ell c_{\la\setminus\la_{j},\la_{j}}+c_{\la\setminus\la_{i}\setminus\la_{j},\la_{i}+\la_{j}}=0,
\ i,j=1,2,\cdots,k_{\la}, i\neq j,
$$$$\label{dd}c_{\la\setminus\la_{i},\la_{i}}(-\cc+\la_{i}\ell)-\sum_{j=1}^{k_{\la}}
c_{\la\setminus\la_{i}\setminus\la_{j},\la_{i}+\la_{j}}=0,
\ i=1,2,\cdots,k_{\la}
$$
generates a non-trivial submodule of $V_{\H}(\ell,\al,\be,\ga)$,
 where in (i) and (ii), $P(m)$ is the set of all partitions of weight $m$,
$\la^{(k)}=(\la,\la,\cdots,\la)\in \Z_{+}^{k}$,
$\la=(\la_{1}^{(q_{1})}, $ $\la_{2}^{(q_{2})},\cdots,
\la_{k_{\la}}^{(q_{k_{\la}})})$ such that
$\sum_{i=1}^{k_{\la}}q_{i}\la_{i}=m$, and
$\la\setminus\la_{i}=(\la_{1}^{(q_{1})},\cdots,
\la_{i}^{(q_{i}-1)},\cdots, \la_{k_{\la}}^{(q_{k_{\la}})})$ and if
$\la=\la_{i}$, then
$c(-\la\setminus\la_{i})b(-\la_{i})=b(-\la_{i})$,
$c(-\la\setminus\la_{i})a(-\la_{i})=a(-\la_{i})$;

\item[(iii)] If $\ell=0$, then for each $\la\in{\mathcal P}$,
$u=c(-\la)v_{0} $
generates a non-trivial submodule of $V_{\H}(\ell,\al,\be,\ga)$.
\end{itemize}
\end{theo}
\qed

\section{\bf Vertex operator algebra structure associated to $\H$}
We assume that the reader is familiar with the basic knowledge on the notations of
vertex operator algebras and their weak modules, admissible modules and ordinary modules.

For $h\in H_{4}$ we define the generating function
$$
h(x)=\sum_{n\in\bZ}(h\ot t^n)x^{-n-1}\in\H[[x,x^{-1}]].
$$
Then the defining relations (\ref{DefinitonANW})  can be equivalently written as
\begin{eqnarray}
[h_1(x_1),h_2(x_2)]=[h_1,h_2](x_2)x_2^{-1}\de\left(\frac{x_1}{x_2}\right)
-(h_1,h_2)\frac{\pa}{\pa x_1}x_2^{-1}\de\left(\frac{x_1}{x_2}\right)\k.\label{CM}
\end{eqnarray}
Given an $\H$-module $W$, let $h(n)$ denote the operator on $W$
corresponding to $h\ot t^{n}$ for $h\in H_{4}$ and $n\in\bZ$.  We
shall use the notation for  the action of $h(x)$ on $W$:
\begin{equation}
h_W(x)=\sum_{n\in\bZ}h(n)x^{-n-1}\in(\End W)[[x,x^{-1}]].
\end{equation}
\begin{defi}
Let $W$ be a {\it restricted}  $\H$-module  in the sense that for
every  $h\in H_{4}$ and $w\in W$, $h(n)w=0$ for $n$ sufficiently
large. We say that the $\H$-module $W$ is of level $\ell$ if the
central element $\k$ acts as a scalar $\ell$ in $\bC$.
\end{defi}

Let  $\ell$ be a complex number. Consider the
induced module defined as (\ref{Vac})(let $\d=0$):
$$
V_{\H}(\ell,0)=U(\H)\ot_{U(\H^{(\leq 0)})}v_{0}.
$$
 Set
$$
\1=v_0\in V_{\H}(\ell,0).
$$
Then
$$
V_{\H}(\ell,0)=\coprod_{n\geq0}V_{\H}(\ell,0)_{(n)},
$$
where $V_{\H}(\ell,0)_{(n)}$ is spanned by  the vectors
$$
h^{(1)}_{-m_1}\cdots h^{(r)}_{-m_r}\1
$$
for $r\geq0,\ h^{(i)}\in H_{4},\ m_i\geq1$, with $n=m_1+\cdots+m_r$. It
is clear that $V_{\H}(\ell,0)$ is a restricted $\H$-module of level $\ell$. We can regard $\H$ as a subspace
 of $V_{\H}(\ell,0)$ through the map
 $$
\H\to V_{\H}(\ell,0), \quad h\mapsto h(-1)\1.
 $$
In fact, $\H=V_{\H}(\ell,0)_{(1)}.$

\begin{theo}[cf. \cite{LL,Lian}]\label{VA} Let $\ell$ be any complex number. Then
there exists  a unique vertex  algebra structure $(V_{\H}(\ell,0),
Y, \1)$ on $V_{\H}(\ell,0)$ such that $ \1 $ is the vacuum vector
and
$$
Y(h,x)=h(x)\in (\End V_{\H}(\ell,0)  )[[x,x^{-1}]]
$$
for $h\in H_{4}$. For $r\geq0, h^{(i)}\in\H, n_i\in\bZ_{+}$, the
vertex operator map for this vertex  algebra structure is given by
\begin{eqnarray*}
Y(h^{(1)}(n_1)\cdots
h^{(r)}(n_r){\bf1},x)
&=&\NO\pa^{(-n_1-1)}h^{(1)}(x)\cdots
\pa^{(-n_r-1)}h^{(r)}(x)\NO1,
\end{eqnarray*}
where
$$
\partial^{(n)} = \frac{1}{n!}
\left(
\frac{d}{dx}
\right)^n
$$
$\NO\NO$ is the normal ordering,
and
$1$ is the identity operator on $V_{\H}(\ell,0)$.
\end{theo}

\begin{pro}[cf. \cite{LL}]\label{VAM}
Any module $W$ for
the vertex  algebra $V_{\H}(\ell,0)$ is naturally a
restricted  $\H$-module of level $\ell$, with $h_{W}(x)=Y_W(h,x)$ for $h\in H_{4}$. Conversely,
any restricted $\H$-module $W$ of level $\ell$  is naturally
a $V_{\H}(\ell,0)$-module as vertex algebra with
$$
Y_W(h^{(1)}(n_1)\cdots
h^{(r)}(n_r){\bf1},x)=\NO\pa^{(-n_1-1)}h^{(1)}_W(x)\cdots
\pa^{(-n_r-1)}h^{(r)}_W(x)\NO1_W,
$$
for $r\geq0, h^{(i)}\in\H, n_i\in\bZ$. Furthermore,
for any $V_{\H}(\ell,0)$-module $W$, the $V_{\H}(\ell,0)$-submodules of $W$ coincide with the
$\H$-submodules of $W$.
\end{pro}
\begin{rem}
Theorem \ref{VA} and Proposition \ref{VAM} in fact hold for the general quadratic Lie algebra , i.e.,
 a (possibly infinite-dimensional) Lie algebra equipped with a symmetric invariant
bilinear form.
\end{rem}

In the following, we shall show that $V_{\H}(\ell,0)$ is in fact a vertex operator algebra under
certain conditions.

Let $\ell$ be a non-zero complex number. Set
\begin{eqnarray}
\om&=&\frac{1}{2\ell}\left(a(-1)b(-1)+b(-1)a(-1)+c(-1)d(-1)+d(-1)c(-1)\right)\1-\frac{1}{2\ell^2}c(-1)c(-1)\1\nonumber\\
&=&\frac{1}{\ell}\left(a(-1)b(-1)\1+c(-1)d(-1)\1\right)-\frac{1}{2\ell}c(-2)\1-\frac{1}{2\ell^2}c(-1)c(-1)\1\label{CV}
\end{eqnarray}
and define operators $L(n)$ for $n\in\bZ$ by
$$
Y(\om,x)=\sum_{n\in\bZ}\om_nx^{-n-1}=\sum_{n\in\bZ}L(n)x^{-n-2}.
$$

Next, we will follow \cite{DL} in using the vertex algebra structure to establish the Virasoro algebra relations,
rather than directly calculating the commutators.

\begin{pro}
Let $\ell$ be a complex number such that $\ell\neq0$. Then for $h\in H_4$ and $m,n\in\bZ$,
\begin{eqnarray}
&&[L(m),h(n)]=-nh(m+n),\label{DG}\\
&&[L(m),L(n)]=(m-n)L(m+n)+\frac{1}{3}(m^3-m)\de_{m+n,0},\label{VR}
\end{eqnarray}
on any restricted $\H$-module $W$ of level $\ell$. In particular,
these relations hold on $V_{\H}(\ell,0)$ and
\begin{eqnarray*}
&&L(0)v=nv,\quad\text{for}\ v\in V_{\H}(\ell,0),\ n\geq0,\\
&&L(-1)=\mathcal{D},
\end{eqnarray*}
where $\mathcal{D}$ is the $\mathcal{D}$-operator of the vertex algebra $V_{\H}(\ell,0)$ defined by
$$
\mathcal{D}v=v_{-2}\1\quad\text{for}\ v\in V_{\H}(\ell,0).
$$
\end{pro}
\begin{proof}
By Theorem \ref{VA}, any restricted $\H$-module $W$ of level $\ell$
is naturally a $V_{\H}(\ell,0)$-module. Then relation (\ref{DG}) can
be written as
\begin{equation}
[h(n),L(m)]=nh(m+n), \end{equation}
for \ $h\in H,\ m,n\in\bZ$.
Equivalently, in terms of generating function, we have
\begin{equation}
[Y(h,x_1),Y(\om,x_2)]=-h(x_2)\frac{\pa}{\pa x_1}x_2^{-1}\de\left(\frac{x_1}{x_2}\right).
\end{equation}
By the commutator formula (\cite{LL}) for the vertex algebra and modules, it suffices to prove
\begin{equation}
h_n\om=\de_{n,1}h,\quad\text{for}\ n\geq0.
\end{equation}
Since $V_{\H}(\ell,0)$ is $\bZ$-graded $\H$-module with $V_{\H}(\ell,0)_{(n)}=0$ for $n< 0$, then
\begin{equation}
h(m)\om=0,\quad\text{for}\ m>2.
\end{equation}
Next we compute $h(2)\om, h(1)\om,$ and $ h(0)\om$, using the relation (\ref{CM}) with $\k$ acting as $\ell$.

\begin{eqnarray*}
\ell h(2)\om&=&h(2)\left(a(-1)b(-1)\1+c(-1)d(-1)\1-\frac{1}{2} c(-2)\1-\frac{1}{2\ell} c(-1)c(-1)\1\right)\\
&=&(a(-1)h(2)+[h,a](1))b(-1)\1+(c(-1)h(2)+[h,c](1))d(-1)\1\\
&-&\frac{1}{2}(c(-2)h(2)+2(h,c)\ell)\1-\frac{1}{2\ell}c(-1)c(-1)h(2)\1\\
&=&[h,a](1)b(-1)\1-( h,c)\ell\1\\
&=&([h,a],b)\ell\1-(h,c)\ell\1\\
&=&0.
\end{eqnarray*}

\begin{eqnarray*}
\ell h(1)\om&=&h(1)\left(a(-1)b(-1)\1+c(-1)d(-1)\1-\frac{1}{2} c(-2)\1-\frac{1}{2\ell} c(-1)c(-1)\1\right)\\
&=&(a(-1)h(1)+[h,a](0)+(h,a)\k)b(-1)\1+(c(-1)h(1)+(h,c)\k)d(-1)\1\\
& &-\frac{1}{2\ell}(c(-1)h(1)+(h,c)\k)c(-1)\1\\
&=&a(-1)(h, b)\ell\1+[[h,a],b](-1)\1+b(-1)(h,a)\ell\1+c(-1)(h,d)\ell\1\\
&&+d(-1)(h,c)\ell\1- c(-1)(h,c)\1\\
&=&\ell h(-1)\1\\
&=&\ell h.
\end{eqnarray*}

\begin{eqnarray*}
\ell h(0)\om&=&h(0)\left(a(-1)b(-1)\1+c(-1)d(-1)\1-\frac{1}{2} c(-2)\1-\frac{1}{2\ell} c(-1)c(-1)\1\right)\\
&=&(a(-1)h(0)+[h,a](-1))b(-1)\1+c(-1)h(0)d(-1)\1-\frac{1}{2\ell}(c(-1)c(-1)h(0)\1\\
&=&a(-1)[h,b](-1)\1+[h,a](-1)b(-1)\1+c(-1)[h,d](-1)\1\\
&=&0.
\end{eqnarray*}
For $h\in H, \ n\in\bZ$, we have
\begin{equation}
[L(-1)-\mathcal{D},h(n)]=0
\end{equation}
as operators on $V_{\H}(\ell,0)$ and
\begin{equation}
(L(-1)-\mathcal{D})\1=\om_0\1-\1_{-2}\1=0.
\end{equation}
It follows that $L(-1)=\mathcal{D}$ on $V_{\H}(\ell,0)$. Similarly, $L(0)=D$, where the weight operator
$D$ is defined by $D v=n v$ for $v\in V_{\H}(\ell,0)_{(n)}$ with $n\in\bZ$. For the Virasoro relations (\ref{VR}), it suffices to prove
\begin{eqnarray*}
&&\om_1\om=L(0)\om=2\om,\\
&&\om_3\om=L(2)\om=2\1,\\
&&\om_n\om=L(n-1)\om=0,\\
\end{eqnarray*}
for $n=2,\ n\geq4.$ By (\ref{DG}), we get
\begin{eqnarray*}
L(0)\om&=&\frac{1}{\ell}L(0)\left(a(-1)b(-1)\1+c(-1)d(-1)\1-\frac{1}{2}c(-2)\1-\frac{1}{2\ell} c(-1)c(-1)\1\right)\\
&=&2\om.\\
L(2)\om&=&\frac{1}{\ell}L(2)\left(a(-1)b(-1)\1+c(-1)d(-1)\1-\frac{1}{2}c(-2)\1-\frac{1}{2\ell} c(-1)c(-1)\1\right)\\
&=&\frac{1}{\ell}\left([L(2),a(-1)]b(-1)\1+[L(2),c(-1)]d(-1)\1-c(0)\1-\frac{1}{2\ell} [L(2),c(-1)]c(-1)\1\right)\\
&&-\frac{1}{\ell}\left(a(-1)b(1)\1+c(-1)d(1)\1-\frac{1}{2\ell} c(-1)c(1)\1\right)\\
&=&\frac{1}{\ell}\left(a(1)b(-1)\1+c(1)d(-1)\1-c(0)\1-\frac{1}{2\ell} c(1)c(-1)\1\right)\\
&=&\frac{1}{\ell}\left((a,b)\k\1+(c,d)\k\1\right)\\
&=&2\1.\\
\end{eqnarray*}
\begin{eqnarray*}
L(1)\om&=&\frac{1}{\ell}L(1)\left(a(-1)b(-1)\1+c(-1)d(-1)\1-\frac{1}{2}c(-2)\1-\frac{1}{2\ell} c(-1)c(-1)\1\right)\\
&=&\frac{1}{\ell}\left(a(0)b(-1)\1+c(0)d(-1)\1-c(-1)\1-\frac{1}{2\ell} c(0)c(-1)\1\right)\\
&=&\frac{1}{\ell}\left(c(-1)\1-c(-1)\1\right)\\
&=&0.\\
L(n)\om&=&\frac{1}{\ell}L(n)\left(a(-1)b(-1)\1+c(-1)d(-1)\1-\frac{1}{2}c(-2)\1-\frac{1}{2\ell} c(-1)c(-1)\1\right)\\
&=&\frac{1}{\ell}\left(a(n-1)b(-1)\1+c(n-1)d(-1)\1-c(n-2)\1-\frac{1}{2\ell} c(n-1)c(-1)\1\right)\\
&=&0
\end{eqnarray*}
for $n\geq3$.
\end{proof}
\begin{rem}
The construction of the conformal vector $\om$, named Nappi-Witten construction,  is different from the Segal-Sugawara construction for
affine Lie algebras \cite{LL, Wang}, because it is easy to check  that  the action of the Casimir element $\Om$ defined in (\ref{Ca})
 acts on $H_4$(under the adjoint representation) is not a scalar. For more general study on conformal vectors,  we refer the reader to \cite{Lian}.
\end{rem}

Summarizing, we have

\begin{theo}
Let $\ell$ be a complex number such that $\ell\neq0$. Then the vertex algebra $V_{\H}(\ell,0)$
constructed in Theorem \ref{VA} is a vertex operator algebra of central charge $4$ with $\om$ defined
in (\ref{CV}) a confomral vector. The $\bZ$-grading on $V_{\H}(\ell,0)$ is given by $L(0)$-eigenvalues.
Moreover, $H=V_{\H}(\ell,0)_{(1)}$, which generates $V_{\H}(\ell,0)$ as vertex algebra, and
$$
[h(n),L(m)]=nh(m+n),\quad\text{for}\ h\in H,\ m,n\in\bZ.
$$
\end{theo}
Let $\ell$ be a complex number such that $\ell\neq0$ and $M$ be a
$H_{4}$-module on which the modified Casimir operator
$\tilde\Om_{\ell}$ acting as a scalar $c_M$. Let $W=\Ind_{\H^{(\geq
0)}}^{\H}(M)$. Since $W$ is a restricted $\H$-module of level
$\ell$, by Proposition \ref{VAM}, $W=\Ind_{\H^{(\geq 0)}}^{\H}(M)$
has a unique admissible module structure for the vertex operator
algebra $V_{\H}(\ell,0)$ such that $Y_W(h,x)=h_W(x)$ for $h\in
H_{4}$. Moreover, $W=\coprod_{n\in\bN}W_{(r+n)}$ with $W_{(r)}=M$,
where $r=\frac{c_M}{2\ell}$. In particular, if $M$ is
finite-dimensional, $W$ is an ordinary module for the vertex
operator algebra $V_{\H}(\ell,0)$.

\begin{theo}
For $\ell\neq0$ and $\d\in\bC$,  the $\H$-module $V_{\H}(\ell,\d)$
is naturally an irreducible ordinary module  for the vertex operator
algebra $V_{\H}(\ell,0)$. Furthermore, the modules $V_{\H}(\ell,\d)$
exhaust all the irreducible ordinary $V_{\H}(\ell,0)$-modules
 up to equivalence.
\end{theo}
\begin{proof}
This proof is completely parallel to the proof of Theorem 6.2.33 in \cite{LL}.
\end{proof}
\begin{rem}
For any vertex operator algebra $V$, Zhu in \cite{Z} constructed an associative
algebra $A(V)$ such that there is one-to-one correspondence between the irreducible admissible $V$-modules
and the irreducible $A(V)$-modules. This fact has been used to classify the irreducible
modules for  vertex operator algebras associated to affine Lie algebra (cf.\cite{FZ}). Similarly,
for $\ell\in\bC^*$, one can show that  $A(V_{\H}(\ell,0))$ is canonically
isomorphic to $U(H_4)$.
\end{rem}

\section{\bf Wakimoto type realizations}
In this section, we shall construct Wakimoto type modules for affine Nappi-Witten algebra $\H$ in terms of vertex operator algebras and their modules.
\subsection{Weyl algebra}
Let $\A$ be the Weyl algebra   with generators
$\be(n),\ga(n)(n\in\bZ)$, ${\bf k}$,  and the following relations
\begin{eqnarray}
[\be(m),\ga(n)]=\de_{m+n,0}{\bf k},\quad
[\ga(m),\ga(n)]=[\be(m),\be(n)]=0, \ [{\bf k}, {\A}]=0.
\end{eqnarray}
Consider the following  irreducible $\A$-module $V_{\A}$ generated
by a vector $\1$ which satisfies:
$$
{\bf k}|_{V_{\A}}={\rm id}, \ \ \be(n)\1=0,\ n\geq0,\quad
\ga(n)\1=0,\ n>0.
$$
Define a linear operator $D$ on $V_{\A}$ by the formulas
$$
D\1=0,\quad [D,\be(n)]=-n\be({n-1}),\quad [D,\ga(n)]=-(n-1)\ga(n-1).
$$
Let
$$
\be(x)=\sum_{n\in\bZ}\be(n)x^{-n-1},\quad \ga(x)=\sum_{n\in\bZ}\ga(n)x^{-n},
$$
Then
\begin{eqnarray*}
[\ga(x_1),\be(x_2)]
=\sum_{m,n\in\bZ}[\ga(m),\be(n)]x_1^{-m}x_{2}^{-n-1}=-x_2^{-1}\de\left(\frac{x_1}{x_2}\right).
\end{eqnarray*}

\begin{theo}[see \cite{FZ}]
There exists  a unique vertex  algebra structure $(V_{\A}, Y, \1)$
on $V_{\A}$ such that $ \1 $ is the vacuum vector and the vertex
operator map for this vertex  algebra structure is given by
$$
Y(\be(-1)\1,x)=\be(x), \quad Y(\ga(0)\1,x)=\ga(x)\in \End
V_{\A}[[x,x^{-1}]]
$$
and
\begin{eqnarray*}
&&Y(\be(-n_1)\cdots \be(-n_r)\ga(-m_1)\cdots
\ga(-m_s)\1,x)\\
&=&\NO\pa^{(n_1-1)}\be(x)\cdots \pa^{(n_r-1)} \be(x)\pa^{(m_1)}
\ga(x)\cdots \pa^{(m_s)} \ga(x)\NO1.
\end{eqnarray*}
for $r,s\geq0, m_i\geq0, n_i\geq1$, where $1$ is the identity
operator on $V_{\A}$.
\end{theo}
\begin{rem}
$V_{\A}$ is not a vertex  operator algebra since it has
infinite dimensional homogeneous components.
\end{rem}

\subsection{ Heisenberg algebras}
Let $\mh$ be a finite dimensional abelian Lie algebra with a
nondegenerate symmetric bilinear form $(\cdot,\cdot)$ and $\hat
{\frak h}={\frak h}\otimes \bC[t,t^{-1}]\oplus \bC\,\c$ the
corresponding affine Lie algebra. Let $\la\in \mh$ and consider the
induced $\hat {\frak h}$-module
$$M(1,\la)=
U(\hat{\frak h})\otimes_{U({\frak h}\otimes{{\C}}[t]
\oplus{{\bC}}\c)}{{\bC}}\simeq S(\frak h\otimes t^{-1}\C[t^{-1}])\ \
\ (\mbox{as vector spaces})
$$
where ${\frak h}\otimes t{{\C}}[t]$ acts trivially on $\C,$
${\frak h}$ acts as $(\alpha,\la)$ for $\alpha\in{\frak h}$ and $\c$
acts as 1.
For $\alpha\in{\frak h}$ and $n\in {\Z}$,  we write $\alpha(n)$
for the operator $\alpha\otimes t^n$ and put
$$\alpha(z)=\sum_{n\in{{\Z}}}\alpha(n)z^{-n-1}.$$
Set ${\bf 1}=1\in\C$. For $\alpha_1, \cdots, \alpha_k \in {\frak
h},\ n_1, \cdots, n_k \in{\Z}_{+}$ and $v=\alpha_1(-n_1)\cdot
\cdot\cdot\alpha_k(-n_k){\bf 1}\in M(1)=M(1,0)$, we define a vertex
operator corresponding to $v$ by
$$Y(v,x)=\mbox{$\circ\atop\circ$}
\partial^{(n_1-1)}\alpha_1(x)\partial^{(n_2-1)}\alpha_2(x)\cdots
\partial^{(n_k-1)}\alpha_k(x)\mbox{$\circ\atop\circ$}.
$$
We can extend $Y$ to all $v\in V$ by linearity. Let
$\{u_1,\cdots,u_d\}$ be an orthonomal basis of ${\frak h}.$ Set
$\omega_{H}=\frac{1}{2}\sum_{i=1}^du_i(-1)^2{\bf 1}$.  The following
theorem is well known :
\begin{theo}[cf. \cite{FLM}]
The space $M(1)=(M(1,0),Y,{\bf 1},\omega_{H})$ is a simple vertex
operator algebra and $M(1,\al)$ for $\al\in {\frak h}$ give a
complete list of inequivalent irreducible modules for $M(1).$
\end{theo}
In the following, we always assume that  $\mh=\bC p\oplus \bC q$ is
a $2$-dimensional abelian Lie algebra equipped with the following
symmetric bilinear form $(\ ,\ )$ such  that
$$
(p,q)=1, \quad (p,p)=(q,q)=0.
$$
Set
$$
p(x)=\sum_{n\in\bZ}p(n)x^{-n-1},\quad q(x)=\sum_{n\in\bZ}q(n)x^{-n-1}.
$$
Then
\begin{eqnarray*}
[p(x_1),q(x_2)]&=&-\frac{\pa}{\pa
x_1}x_2^{-1}\de\left(\frac{x_1}{x_2}\right).
\end{eqnarray*}

\subsection{Wakimoto type realization}
\begin{theo}\label{theorem4.4}
Let $\ell$ be a nonzero complex number. Then there exists a
homomorphism of vertex algebras
$$
\Phi: V_{\H}(\ell,0)\lra V_{\A}\ot M(1)\\
$$
defined by
\begin{eqnarray*}
&&a(x)\mapsto  \be(x),\\
&&b(x)\mapsto \ell\ga'(x)+p(x)\ga(x),\\
&&c(x)\mapsto  p(x),\\
&&d(x)  \mapsto \ell q(x)+\frac{1}{2}\ell^{-1} p(x)-\NO\be(x)\ga(x)\NO.\\
\end{eqnarray*}

\end{theo}
\begin{proof}
It suffices to prove the following commutation relations:
\begin{eqnarray*}
&&[d(x_1),d(x_2)]\\
&=&[\ell
q(x_1)+\frac{1}{2}\ell^{-1}p(x_1)-\NO\be(x_1)\ga(x_1)\NO,\ell
q(x_2)+\frac{1}{2}\ell^{-1}p(x_2)
-\NO\be(x_2)\ga(x_2)\NO]\\
&=&\frac{1}{2}[q(x_1),p(x_2)]+\frac{1}{2}[p(x_1),q(x_2)]+[\NO\be(x_1)\ga(x_1)\NO,\NO\be(x_2)\ga(x_2)\NO]\\
&=&-\frac{\pa}{\pa x_1}x_2^{-1}\de\left(\frac{x_1}{x_2}\right)+\frac{\pa}{\pa x_1}x_2^{-1}\de\left(\frac{x_1}{x_2}\right)\\
&=&0.
\end{eqnarray*}

\begin{eqnarray*}
[a(x_1),b(x_2)]
&=&[\be(x_1),\ell\ga'(x_2)+ p(x_2)\ga(x_2)]\\
&=&\ell[\be(x_1),\ga'(x_2)]+ p(x_2)[\be(x_1),\ga(x_2)]\\
&=& -\ell\frac{\pa}{\pa x_1}x_2^{-1}\de\left(\frac{x_1}{x_2}\right)
+p(x_2)x_2^{-1}\de\left(\frac{x_1}{x_2}\right)\\
&=& c(x_1)x_2^{-1}\de\left(\frac{x_1}{x_2}\right)-\ell\frac{\pa}{\pa
x_1}x_2^{-1}\de\left(\frac{x_1}{x_2}\right).
\end{eqnarray*}

\begin{eqnarray*}
[c(x_1),d(x_2)]&=&[p(x_1),\ell q(x_2)+\frac{1}{2}\ell^{-1}p(x_2)-\NO\be(x_2)\ga(x_2)\NO]\\
&=&\ell[p(x_1),q(x_2)] =-\ell \frac{\pa}{\pa
x_1}x_2^{-1}\de\left(\frac{x_1}{x_2}\right).
\end{eqnarray*}

\begin{eqnarray*}
&&[d(x_1),a(x_2)]\\
&=&[\ell q(x_1)+\frac{1}{2}\ell^{-1}p(x_1)-\NO\be(x_1)\ga(x_1)\NO,\be(x_2)]\\
&=& -a(x_1)[\ga(x_1),\be(x_2)]\\
&=&a(x_2)x_2^{-1}\de\left(\frac{x_1}{x_2}\right).
\end{eqnarray*}

\begin{eqnarray*}
&&[d(x_1),b(x_2)]\\
&=&[\ell q(x_1)+\frac{1}{2}\ell^{-1}p(x_1)-\NO\be(x_1)\ga(x_1)\NO,\ell\ga'(x_2)+ p(x_2)\ga(x_2)]\\
&=&\ell[q(x_1),p(x_2)]\ga(x_2)-\ell[\be(x_1),\ga'(x_2)]\ga(x_1)-\ga(x_1)p(x_2)[\be(x_1),\ga(x_2)]\\
&=&-\ell\frac{\pa}{\pa
x_1}x_2^{-1}\de\left(\frac{x_1}{x_2}\right)\ga(x_2)
+\ell\left(\frac{\pa}{\pa x_1}x_2^{-1}\de\left(\frac{x_1}{x_2}\right)\right)\ga(x_1)-\ga(x_1)p(x_2)[\be(x_1),\ga(x_2)]\\
&=&-\ell\frac{\pa}{\pa
x_1}x_2^{-1}\de\left(\frac{x_1}{x_2}\right)\ga(x_2)
+\ell\frac{\pa}{\pa
x_1}x_2^{-1}\de\left(\frac{x_1}{x_2}\right)\ga(x_2)-\ell
x_2^{-1}\de\left(\frac{x_1}{x_2}\right)\ga'(x_2)\\
&&-\ga(x_1)p(x_2)x_2^{-1}\de\left(\frac{x_1}{x_2}\right)\\
&=&-(\ell\ga'(x_2)+\ga(x_2)p(x_2))
x_2^{-1}\de\left(\frac{x_1}{x_2}\right)\\
&=&-b(x_2)x_2^{-1}\de\left(\frac{x_1}{x_2}\right).
\end{eqnarray*}

Here we use the following fact:

\begin{eqnarray*}
&&\left(\frac{\pa}{\pa x_1}x_2^{-1}\de\left(\frac{x_1}{x_2}\right)\right)\ga(x_1)\\
&=&\ga(x_2)\frac{\pa}{\pa
x_1}x_2^{-1}\de\left(\frac{x_1}{x_2}\right)-
\ga'(x_2)x_2^{-1}\de\left(\frac{x_1}{x_2}\right).
\end{eqnarray*}

\begin{eqnarray*}
[c(x_1),a(x_2)]=[c(x_1),b(x_2)]=[p(x_1),\be(x_2)]=0.
\end{eqnarray*}

\end{proof}

\begin{rem} \ (1) $
\Phi$ is not surjective, since $\ga(0){\bf 1}$ has no preimage.

(2) \  In some physics literature\cite{CFS,KK}, several different
Wakimoto type modules for $\H$ were given.
\end{rem}

As a consequence, we have the{\it Wakimoto type modules} over $\H$:

\begin{cor}
For $\al\in\mh$, $V_{\A}\ot M(1,\al)$ is a module for the vertex
algebra $V_{\H}(\ell,0)$ and an $\H$-module of level $\ell$.
\end{cor}
\begin{proof}
Since $M(1,\al)$ is a module for $M(1,0)$, it follows that
$V_{\A}\ot M(1,\al)$ is naturally a module for the vertex algebra
$V_{\A}\ot M(1,0)$. Then by  Theorem \ref{theorem4.4}, $V_{\A}\ot
M(1,\al)$ is a $V_{\H}(\ell,0)$-module. By Proposition \ref{VAM},
$V_{\A}\ot M(1,\al)$ is also an $\H$-module of level $\ell$.

\end{proof}

\begin{rem}\label{remark4.4}
If $(\alpha, p)=0$, then $V_{\A}\ot M(1,\al)$ contains an
irreducible submodule ismorphic to $V_{\H}(\ell,{\d})$ for some
${\d}\in\C$. If $(\alpha, p)\neq 0$, then $V_{\A}\ot M(1,\al)$ is
isomorphic to $V_{\H}(\ell, {\cc}, {\d})$ for some ${\cc}\in\C^{*}$,
${\d}\in\C$.
\end{rem}

\end{document}